\documentclass[a4paper, 11pt]{article}
\usepackage{amsfonts,latexsym,rawfonts,amsmath,amssymb,amsthm}

\usepackage[applemac]{inputenc}
\usepackage[all]{xy}
\usepackage[T1]{fontenc}
\usepackage{textcomp}
\usepackage{geometry}
\usepackage{graphicx}
\usepackage{mathptmx}
\usepackage{hyperref}

\numberwithin{equation}{section}

\def\pt{{\sqrt{-1}\partial_t\bar\partial_t}}
\newcommand{\pd}[2]{\frac {\partial #1}{\partial #2}}

\newcommand{\al}{\alpha}
\newcommand{\bb}{\beta}

\newcommand{\oo}{\omega}

\newcommand{\Na}{\nabla}

\newcommand{\ee}{\epsilon}

\newcommand{\Te}{\Theta}
\newcommand{\te}{\theta}

\newcommand{\beq}{\begin{equation}}
\newcommand{\eeq}{\end{equation}}
\newcommand{\beqs}{\begin{eqnarray*}}
\newcommand{\eeqs}{\end{eqnarray*}}
\newcommand{\beqn}{\begin{eqnarray}}
\newcommand{\eeqn}{\end{eqnarray}}
\newcommand{\beqa}{\begin{array}}
\newcommand{\eeqa}{\end{array}}

\def\td{\tilde}
\def\p{\partial}

\def\s{{\bf{s}}}

\def\RR{{\mathbb R}}

\def\PP{{\mathbb P}}
\def\CC{{\mathbb C}}

\def\LL{{\mathbb L}}

\def\GG{{\mathbb G}}

\def\ri{\rightarrow}

\def\un{\underline}

\def\pbp{\sqrt{-1}\partial\bar\partial}
\def\tr{{\rm tr}}

\def\ba{{\mathbf a}}

\def\cB{{\mathcal B}}

\def\cH{{\mathcal H}}

\def\cL{{\mathcal L}}

\def\cP{{\mathcal P}}

\def\cU{{\mathcal U}}

\def\i{{\sqrt{-1}}}

\def\Aut{{\rm Aut}}

\def\ba{\Xi}

\newtheorem{prop}{Proposition}[section]
\newtheorem{theo}[prop]{Theorem}
\newtheorem{lem}[prop]{Lemma}

\newtheorem{cor}[prop]{Corollary}
\newtheorem{rem}[prop]{Remark}
\newtheorem{ex}[prop]{Example}

\title{ Complex deformation of critical K\"ahler metrics}
\author{Haozhao Li\footnote{Research supported in part by National Science Foundation
of China No. 11001080 and No. 11131007.}}

\begin{document}
\bibliographystyle{plain}

\date{}

\maketitle

\tableofcontents

\section{Introduction}
In \cite{[Calabi1]}\cite{[Calabi2]}, Calabi introduced the extremal
K\"ahler metrics, which is the citical point of the $L^2$ norm of
the scalar curvature in the K\"ahler class. The existence and
uniqueness of the extremal K\"ahler metrics have been intensively
studied during past decades(cf. \cite{[APS]}\cite{[CT]} and
reference therein). By Kodaira-Spencer's work \cite{[KS]}, every
K\"ahler manifold admits K\"ahler metrics under small perturbation
of the complex structure. A natural question is whether
K\"ahler-Einstein metrics or extremal K\"ahler metrics still exist
when the complex structures varies. In \cite{[Koiso]}, Koiso showed
that the K\"ahler-Einstein metrics can be perturbed under the
complex deformation of the complex structure when the first Chern
class is zero or negative. When the first Chern class is positive,
Koiso showed this result if the manifold has no nontrivial
holomorphic vector fields. In \cite{[LB1]}\cite{[LB2]},
Lebrun-Simanca systematically studied the deformation theory of
extremal K\"ahler metrics and constant scalar curvature metrics and
they proved that on a K\"ahler manifold, the set of K\"ahler classes
which admits extremal metrics is open and the constant scalar
curvature metrics can be perturbed under some extra  restrictions.
Based on Lebrun-Simanca's results, Apostolov-Calderbank-Gauduchon-T.
Friedman \cite{[ACGT]}, Rollin-Simanca-Tipler
\cite{[RST]}\cite{[RT2]} further discussed extremal metrics under
the deformation of complex structures.

The main goal of this paper is to give an alternative proof on the
deformation of constant scalar curvature metrics, which was
discussed by \cite{[LB1]} in the case of fixed complex structure,
and later by \cite{[ACGT]}\cite{[RST]} in the case of varying
complex structures. Here we use the method of Pacard-Xu in
\cite{[PX]} in the context of constant mean curvature problems,
which is quite different from \cite{[LB1]} in analysis.  We will
also discuss the deformation of K\"ahler-Ricci solitons.\\

First we consider the case of fixed complex structure.  The main
difficulty of the deformation problems of the K\"ahler-Einstein
metrics or constant scalar curvature metrics is that the linearized
equation has nontrivial kernel so that we cannot use the implicit
function theorem directly. For this reason, Koiso in \cite{[Koiso]}
assumed that the manifold has no nontrivial holomorphic vector
fields, and Lebrun-Simanca in \cite{[LB1]} used the surjective
version of the implicit function theorem so that the nondegeneracy
of the Futaki invariant must be assumed. The same difficulty appears
in some other geometrical equations such as the constant mean
curvature equation. In \cite{[PX]}, Pacard-Xu constructed a new
functional to solve the constant mean curvature equation and they
removed the nondegeneracy condition of Ye's result in \cite{[Ye]}.
We observe that Pacard-Xu's method can be applied in our situation
and we have the result:

\begin{theo}\label{theo001}

Let $(M, \oo_g)$ be a compact K\"ahler manifold with a constant
scalar curvature metric $\oo_g.$ There exists $\ee_0>0$ and a smooth
function
$$\Phi: (0, \ee_0)\times \cH^{1, 1}(M)\ri \RR$$
such that if $\bb\in \cH^{1, 1}(M)$ has unit norm and  satisfies
$\Phi(t,
  \bb)=0$ for some $t\in (0, \ee_0)$ then $M$ admits a constant scalar curvature metric
   in the K\"ahler class $[\oo_g+t\bb].$ Moreover,
\begin{enumerate}
  \item[(1)] If $\bb\in \cH^{1, 1}(M)$ is traceless,  $\Phi$ has the expansion:
  $$\Phi(t, \bb)=t^2\int_M\;(\Pi_g(R_{i\bar j}\bb_{j\bar i}))^2\,\oo_g^n+O(t^3).$$
  \item[(2)] If $\bb\in \cH^{1, 1}(M)$ is traceless and $\oo_g$ is a K\"ahler-Einstein metric,
then
  $\Phi$ has the expansion:
  $$\Phi(t, \bb)=t^4\int_M\;(\Pi_g(\bb_{i\bar j}\bb_{j\bar
  i}))^2\,\oo_g^n+O(t^5).$$
\end{enumerate} Here the operator $\Pi_g$ is the projection to the
space of Killing potentials with respect to $\oo_g.$

  \end{theo}

Theorem \ref{theo001} gives us some information in which directions
 we can find the constant scalar curvature metrics. The function $\Phi$ is
constructed by the Futaki invariant, and it is automatically zero
when the Futaki invariant vanishes. Thus, a direct corollary of
Theorem \ref{theo001} is the following result, which was proved by
Lebrun-Simanca using the deformation theory of the extremal K\"ahler
metrics and a result of Calabi in \cite{[Calabi2]}:

\begin{cor}\label{cor001}(Lebrun-Simanca \cite{[LB1]}) Let $(M, \oo_g)$ be a compact K\"ahler manifold with a
constant scalar curvature metric $\oo_g.$ For any $\bb\in \cH^{1,
1}(M)$, there is a $\ee_0>0$ such that if the Futaki invariant
vanishes on the K\"ahler class $[\oo_g+t\bb]$ for some $t\in (0,
\ee_0)$, then $M$ admits a constant scalar curvature  metric on
$[\oo_g+t\bb]$.

\end{cor}

In fact, Theorem \ref{theo001} gives us more information on the
existence of constant scalar curvature metrics on the class
$[\oo_g+t\bb].$ We expand the function $\Phi(t, \bb)$ with respect
to $t$  at $t=0,$
$$\Phi(t, \bb)=\sum_{j=1}^m\,a_j(\bb)t^j+O(t^{m+1}),$$
where $a_j(\bb)$ are some functions of $\bb.$ If we assume some of
$a_j(\bb)$ vanish, then we can get ``almost constant scalar
curvature metrics'' in the following sense:

\begin{cor}\label{cor002} Let $\oo_g$ be a constant scalar curvature metric. There
are two positive constants $\ee$ and $C$ such that for any $\bb\in
\cH^{1, 1}(M) $ with \beq a_1(\bb)=a_2(\bb)=\cdots=a_m(\bb)=0,
\nonumber\eeq $M$ admits a K\"ahler metric $\oo_{t, \bb}\in
[\oo_g+t\bb]$ for $t\in (0, \ee)$
 satisfying \beq \|\s(\oo_{t, \bb})-\un \s(t)\|_{C^k(M)}\leq
Ct^{\frac {m+1}{2}},\nonumber\eeq where $\un \s(t)$ is the average
of the scalar curvature in $[\oo_g+t\bb].$\\

\end{cor}

The case of varying complex structures is more difficult. In general
the extremal metrics may not be perturbed when the complex structure
varies (cf. \cite{[BB]}). There are several results on this problem
recently. In \cite{[ACGT]} Apostolov-Calderbank-Gauduchon-T.
Friedman showed that the extremal metrics can be perturbed when the
deformation of the complex structure is invariant under the action
of a maximal compact connected subgroup $G$ of the isometry group of
the extremal metrics.
 Rollin-Simanca-Tipler extend this result in \cite{[RST]} and they allow
the group $G$ extends partially to the complex deformation.  Here we
combine Rollin-Simanca-Tipler and Pacard-Xu's methods to get a
similar result as in the case of fixed complex structures.

Before stating the next result, we need to introduce some notations.
Let $(M, J, g,  \oo_{g})$ be a compact K\"ahler manifold with
 a constant scalar curvature metric
$(g, \oo_g)$ and  $G $   the identity component of the isometry
group of $(M, g)$. We assume that a compact connected subgroup $G'$
of $G$ acts holomorphically on a complex deformation $(J_t, g_t,
\oo_t)$  and we denote by $\cB_{G'}$ the space of all such complex
deformations. Let $W^{2, k}_{G'}$ be the space of $G'$-invariant
functions in $W^{2, k}$ and $\cH_g^{\frak z_0'}$ be the space of the
space of holomorphic potentials of the elements in the center $\frak
z_0'$ of $\frak g_0'$, where $\frak g_0'$ is the ideal of the
Killing vector fields with zeroes  in the Lie algebra of $G'$. With
these notations, we have

\begin{theo}\label{theo002}   Let $(M, J, g,  \oo_{g})$ be a compact K\"ahler manifold with
 a constant scalar curvature metric
$  \oo_g$ and \beq  \ker\LL_g\cap W_{G'}^{2, k}\subset
\RR\oplus\cH_g^{\frak z_0'}.\label{eq051}\eeq
 For any  $(J_t, g_t, \oo_t) \in \cB_{G'}$,
there is a constant $\ee_0>0$ and a smooth function $\Psi:
\cB_{G'}\ri \RR$ such that if \beq \Psi(J_t, g_t, \oo_t)=0
\label{eq052}\eeq for some $t\in (0, \ee_0)$, then $M$ admits a
$G'$-invariant constant scalar curvature metric in $[\oo_t]$ with
respect to $J_t$.  In particular, the conclusion holds if the
condition (\ref{eq052}) is replaced by the vanishing of the  Futaki
invariant of $[\oo_t]$.
\end{theo}

 The condition (\ref{eq051}) coincides with the non-degeneracy
condition of the relative Futaki invariant, which is introduced by
Rollin-Simanca-Tipler in \cite{[RST]}. Here we get the same
condition from a different point of view.  We can get a similar
result as Corollary \ref{cor002} and a similar expansion of the
function $\Psi$ as in Theorem \ref{theo001}, which are omitted since
we will not use them in this
paper. \\

Finally, we will study the deformation of the K\"ahler-Ricci
soliton. A K\"ahler-Ricci soliton is a K\"ahler metric $\oo_g$ in
the first Chern class satisfying
$$Ric(\oo_g)-\oo_g=\pbp\theta_X,$$
where $\theta_X$ is a holomorphic potential of a holomorphic vector
field $X$. As  K\"ahler-Einstein metrics, the existence and
uniqueness of K\"ahler-Ricci soliton are important and has been
studied by a series of papers \cite{[TZ1]}\cite{[WZ]} etc. Since
K\"ahler-Ricci solitons must be in the first Chern class, there are
no K\"ahler-Ricci solitons if we deform the K\"ahler class. However,
inspired by the extremal K\"ahler metrics, we can consider whether
there is a metric satisfying the equation
$$\s(\oo_g)-\un \s=\Delta_g\theta_X,$$
where $\un \s$ is the average of the scalar curvature $\s.$ This
metric is first introduced by Guan in \cite{[Guan]} and is called
extremal solitons. Using the same idea as in
\cite{[LB1]}\cite{[LB2]}, we have the result:

\begin{theo}\label{theo003}Let $(M, J, g,  \oo_{g})$ be a compact K\"ahler manifold with
 a K\"ahler-Ricci soliton
$(g, \oo_g)$.
\begin{enumerate}
  \item If the complex structure is fixed, for any $\bb\in \cH^{1, 1}(M)$ there is an extremal soliton in the
K\"ahler class $[\oo_g+t\bb]$ for small $t$.
  \item  For any  $(J_t, g_t, \oo_t) \in \cB_{G}$ where $G$ is the identity component of
  the isometry group of   $(M, g),$  $M$
admits a $G$-invariant extremal soliton in  $[\oo_t]$ with respect
to $J_t$.
\end{enumerate}
\end{theo}

Under the assumption of the second part of Theorem \ref{theo003}, if
in addition  $[\oo_t]$ is the first Chern class of $(M, J_t)$, then
$[\oo_t]$ admits a K\"ahler-Ricci soliton. It is interesting to see
whether Theorem \ref{theo003} holds  for any extremal soliton. There
is a technical difficulty
in the proof and we cannot overcome it here. \\

\noindent {\bf Acknowledgements}: The author would like to thank
Professor F. Pacard and Y. L. Shi for kindly sharing their insights
on the deformation theory. The author would also like to thank
Professor X. X. Chen and X. H. Zhu for their encouragement and
numerous suggestions.

\section{Deformation of cscK metrics}\label{Sec001}
In this section, we will use the method of Pacard-Xu in \cite{[PX]}
to solve the constant scalar curvature equation  and show that a
small perturbation of the K\"ahler class under some assumptions will
admit a constant scalar curvature metric.

\subsection{Fixed complex structure}

 We follow Lebrun-Simanca's notations in
\cite{[LB1]}\cite{[LB2]}.  Let $(M, J, g, \oo_g)$ be a compact
K\"ahler manifold of complex dimension $n$ with a constant scalar
curvature metric $\oo_g.$
  By Matsushima-Lichnerowicz theorem,  the identity component $G$
of the isometry group of $(M, g)$ is a maximal compact subgroup of
the identity component $\Aut_0(M, J)$ of  the automorphism group
$\Aut(M, J).$ Let $W^{2, k}_G(M)$ be the real $k$-th Sobolev space
of $G$-invariant real-valued functions in $W^{2, k}(M).$ By the
Sobolev embedding theorem, the space $W^{2, k}(M)$ is contained in
$C^{l}(M)$ if $k>n+l$. The space of real-valued $\oo_g$-harmonic
$(1, 1)$ forms on $M$ is denoted by $\cH^{1, 1}(M).$ Since the
metric $g$ is $G$-invariant, every $g$-harmonic form $\bb\in \cH^{1,
1}(M)$ is $G$-invariant. Let $\cP(M, \oo_g)$ be the space of
K\"ahler potentials of $\oo_g$ and $\cU$ be a small neighborhood of
the origin in $W^{2, k}_G(M)$. We can assume that $\cU\subset \cP(M,
\oo_t)$ for small $t$ where $\oo_t=\oo_g+t\bb$. Thus, for any
function $\varphi\in \cU$ the metric
$$\oo_{t, \varphi}=\oo_g+t\bb+\pbp \varphi,$$
is $G$-invariant.\\

Let $\frak h(M, J)$ be the space of holomorphic vector fields on
$(M, J).$ By Matsushima-Lichnerowicz theorem, the Lie algebra $\frak
h(M, J)$ can be decomposed as a direct sum
$$\frak h(M, J)=\frak h_0(M, J)\oplus \frak a(M, J),$$
where $\frak a(M, J)$ consists of the autoparallel holomorphic
vector fields of $(M, J)$ and $\frak h_0(M, J)$ is the space of
holomorphic vector fields with zeros. Let $\frak g$ the Lie algebra
of $G$ and  $\frak g_0$ the ideal of Killing vector fields with
zeros. Any element $\xi\in \frak g_0$ corresponds to a holomorphic
vector field $X=J\xi+\i \xi,$ and we define a smooth function
$\te_X$ satisfying
$$i_X\oo_g=\i \bar\p \theta_X,\quad \int_M\;\te_X\,\oo_g^n=0.$$
The function $\te_X$ is called holomorphic potential of $X$ with
respect to $\oo_g.$  Since $g$ is $G$-invariant,   $\te_X$ is a
real-valued function. Let $\frak z\subset \frak g$ denote the center
of $\frak g$ and  $\frak z_0=\frak z\cap \frak g_0.$ Then  $\frak
z_0$ corresponds precisely to the Killing vector fields in $\frak
g_0$ whose holomorphic potentials are $G$-invariant.

 Now we choose a basis $\{\xi_1, \cdots,
\xi_d\}$
 of $\frak z_0$ such that the functions $\{\te_0, \te_1, \cdots,
 \te_d\}$, where $\te_0=1$ and $\te_i$ is the holomorphic potential
 of the holomorphic vector fields   $X_i=J\xi_i+\i\xi_i $,   are orthonormal
with respect to the $L^2$ inner product induced by the metric $g$
$$\langle f, g\rangle_{L^2(\oo_g)}=
\frac 1{V_g}\int_M\;fg\,\oo_g^n,\quad f, g\in C^{\infty}(M, \RR),$$
where $V_g=\int_M\;\oo_g^n.$  Using this product, the space $W^{2,
k}_G$ has a decomposition
$$W^{2, k}_G=\cH_g\oplus \cH_{g, k}^{\perp},$$
where $\cH_g$ is spanned by the set $\{\theta_0, \te_1, \cdots,
\te_d\}$ over $\RR.$  We define the associate projection operator
 \beqs
\td\Pi_g: W^{2, k}_G&\ri &\cH_g\\
f&\ri &\sum_{i=0}^d\langle\theta_i, f\rangle_{L^2(\oo_g)}\theta_i,
 \eeqs and the operator $\td \Pi_g^{\perp}=I-\td\Pi_g.$\\

 For any $\varphi\in \cU$, we   calculate the expansion of the
scalar curvature of $\oo_{t, \varphi}$ at $(t, \varphi)=(0, 0):$
$$s(\oo_{t,  \varphi })=s(\oo_g)-\Big(\Delta^2_g\varphi+R_{i\bar j}\varphi_{j\bar i}
+t\Delta_g\tr_{\oo}\bb+t R_{i\bar j}\bb_{j\bar
i}\Big)+Q_g(\Na^2\varphi, t\bb),$$ where $Q_g$ collects all the
higher order terms. Note that $\tr_g\bb$ is a constant since $\bb$
is harmonic. The linearized operator of $s(\oo_{t, \varphi})$ at
$(t, \varphi)=(0, 0)$ is given by
$$\LL_g\varphi=\Delta^2_g\varphi+R_{i\bar j}\varphi_{j\bar
i},
$$ and  for
any $f\in \ker\LL_g$ we can associate a holomorphic vector field
$X_f=J\Na f+\i \Na f $ which has nonempty zeros.  In general,
$\LL_g$ has nontrivial kernel and it is difficult to solve the
constant scalar curvature equation.

Now we have the following result:

\begin{theo}\label{theo014}
Let $(M, \oo_g)$ be a compact K\"ahler manifold with a constant
scalar curvature metric $\oo_g.$ There exists $\ee_0>0$ and a smooth
function
$$\Phi: (0, \ee_0)\times \cH^{1, 1}(M)\ri \RR$$
such that if $\bb\in \cH^{1, 1}(M)$ has unit norm and  satisfies
$\Phi(t,
  \bb)=0$ for some $t\in (0, \ee_0)$ then $M$ admits a constant scalar curvature metric
   in the K\"ahler class $[\oo_g+t\bb].$

  \end{theo}
\begin{proof}
Consider the equation for $(\varphi, \td\Xi)\in \cH_{g,
k}^{\perp}\times \RR^{d+1}:$ \beq s(\oo_{t, \varphi})=\langle \td
\Xi, \td \Te\rangle,\label{eq0010} \eeq where $\td\Te=(\te_0,
\te_1,\cdots, \te_d)$ and $\td \Xi=(c_0, c_1,\cdots, c_d)\in
\RR^{d+1}$ is a vector with
$$\langle \td \Xi, \td \Te\rangle=c_0+\sum_{i=1}^d\,c_i\theta_i.$$
Note that if the equation (\ref{eq0010}) holds, then $c_0$ is the
average of the scalar curvature and it only depends on the K\"ahler
class $[\oo_t].$  Applying the implicit function theorem, we have

\begin{lem}\label{lem2020}Fix $\bb\in \cH^{1, 1}(M)$.  Then there exist   $\ee_0, C>0$ such
 that for
all $t\in (0, \ee_0)$   there exists a unique solution $(\varphi_{t,
\bb}, \td \ba_{t, \bb})\in \cH_{g, k+4}^{\perp}\times \RR^{d+1}$ of
the equation (\ref{eq0010}) and satisfying the estimates \beq
\|\varphi_{t, \bb}\|_{W^{2, k+4}(M)}\leq C\ee_0,\quad \|\td \ba_{t,
\bb}\|\leq C\ee_0,\label{eq012}\eeq where $\|\td \Xi\|$ denotes the
standard Euclidean norm of $\td \Xi$ in $\RR^{d+1}.$

\end{lem}
\begin{proof}We consider the operator
$$ \td\Pi_g^{\perp}s(\oo_{t, \varphi}): (-\ee, \ee)\times \cH_{g, k+4}^{\perp}\ri \RR.$$
 Since the linearized operator at $(t, \varphi)=(0, 0)$
\beqs
 D_{\varphi} \td\Pi^{\perp}s(\oo_{t, \varphi})|_{(0, 0)}: \cH_{g, k+4}^{\perp}&\ri & \cH_{g, k }^{\perp}\\
  \psi &\ri &-\LL_g\psi
\eeqs is invertible, for small $t$ there is a solution $ \varphi_{t,
\bb}\in \cH_{g, k+4}^{\perp}$ such that $\td \Pi_g^{\perp}s(\oo_{t,
\varphi_{t, \bb}})=0$ and we can find a vector $\hat \Xi_{t, \bb}\in
\RR^{d+1}$ such that \beq s(\oo_{t, \varphi_{t, \bb}})=\langle
\td\Xi_{t, \bb}, \td \Te\rangle. \label{eq017}\eeq The  estimates in
(\ref{eq012})
follows  directly from the implicit function theorem. \\

\end{proof}

Now we want to know when the solution $(\varphi_{t, \bb}, \td
\Xi_{t, \bb})$ of (\ref{eq0010}) has  constant scalar curvature. It
suffices to show that the vector $\td \Xi_{t, \bb}=(c_0, c_1,
\cdots, c_d)$ satisfies $c_i=0$ for all $1\leq i\leq d$. Given
 $\bb\in \cH^{1, 1}(M)$, the solution $(\varphi_{t, \bb}, \td \ba_{t,
\bb})$ determines a holomorphic vector field \beq X_{t,
\bb}=\sum_{k=1}^d\,c_k(t) X_k\in \frak h_0(M, J),\label{eq009} \eeq
where $X_k$ is the holomorphic vector field defined by $\theta_k$
and $c_i(t)$ are the entries of the vector $\td \Xi_{t,
\bb}=(c_0(t), c_1(t),\cdots, c_d(t)).$ For simplicity, we write
$\oo_{t, \bb}=\oo_{t, \varphi_{t, \bb}}$ for short.  Now we define a
function on $(0, \ee_0)\times\cH^{1, 1}(M)$ by
$$\Phi(t, \bb)=\int_M\;X_{t, \bb} h_{\oo_{t, \bb}}\,\oo_{t,
\bb}^n,$$ where $h_{\oo_{t, \bb}}$ is determined by $s(\oo_{t,
\bb})-c_0(t) =\Delta_{\oo_{t, \bb}} h_{\oo_{t, \bb}}.$ Note that the
function $\Phi(t, \bb)$ is exactly the Futaki invariant of $(X_{t,
\bb}, [\oo_t])$, and it is zero if the Futaki invariant of $[\oo_t]$
vanishes. Let $\Pi_g$ be the $L^2$-projection from $W_{G}^{2, k}(M)$
to the subspace which is spanned by the functions $\{\te_1, \cdots,
\te_d\}$. We denote by $\Xi_{t, \bb}=(c_1,
\cdots, c_d)$ the vector in $\RR^d$ which removes $c_0$ from $\td \Xi_{t, \bb}$
and $\Te=(\te_1, \cdots, \te_d).$ With these notations, we have the lemma: \\

\begin{lem}\label{lem003} There is a $\ee_0>0$ such that, if $t\in (0,
\ee_0)$ and  if $\bb\in \cH^{1, 1}(M)$ with unit norm is a zero of
the function $\Phi(t, \bb)$ then $\oo_{t, \bb}$ has constant scalar
curvature.

\end{lem}

\begin{proof}Note that
\beq \Phi(t, \bb)=\int_M\;\te_{t, \bb} (s(\oo_{t,
\bb})-c_0(t))\,\oo_{t, \bb}^n= \int_M\;\te_{t, \bb}\langle  \ba_{t,
\bb},
  \Theta\rangle\,\oo_{t, \bb}^n,\label{eq004}\eeq where $\te_{t,
\bb}$ is the holomorphic potential of $ X_{t, \bb}$ with respect to
the metric $\oo_{t, \bb}$ under the normalization condition \beq
\int_M\;\theta_{t, \bb}\,\oo_{t, \bb}^n=0. \label{eq301}\eeq
 We
claim that there is a constant $C$ independent of $t$ and $\bb$ such
that \beq \|\te_{t, \bb}-\langle \ba_{t, \bb},
\Theta\rangle\|_{L^2(\oo_g)}\leq C\;t\|\ba_{t, \bb}\|. \label{eq003}
\eeq In fact, by definition we have
$$  i_{X_{t, \bb}}\oo_g=\i\bar \p \langle   \ba_{t, \bb},   \Theta\rangle,
\quad i_{X_{t, \bb}}\oo_{t, \bb}=\i\bar \p \theta_{t, \bb}.$$ This
implies that \beqs \i\bar \p(\theta_{t, \bb}- \langle   \ba_{t,
\bb}, \Theta\rangle)&=&i_{X_{t, \bb}}(t\bb+\pbp\varphi_{t,
\bb})=\sum_{k=1}^d\;c_k(t)i_{X_k}(t\bb+\pbp\varphi_{t, \bb}),\eeqs
where we used the definition (\ref{eq009}) of $X_{t, \bb}.$ Since by
Lemma \ref{lem2020} $\|\varphi_{t, \bb}\|_{W^{2, k+4}(M)}\leq
C\ee_0$ for any $t\in (0, \ee_0)$, we have \beqs
\Big|\Delta_g(\theta_{t, \bb}- \langle \ba_{t, \bb},
\Theta\rangle)\Big|&=&\Big|\sum_{k}\,c_k(t)\,\tr_g\Big(\p (
i_{X_k}(t\bb+\pbp\varphi_{t,
\bb}))\Big)\Big|\\&\leq&C\,\ee_0\,\|\ba_{t, \bb}\|,\eeqs which
implies that
$$\|\theta_{t, \bb}-
\langle \ba_{t, \bb}, \Theta\rangle\|_{L^2(\oo_g)}\leq
C\,\ee_0\,\|\ba_{t, \bb}\|$$ by the eigenvalue decomposition of
$\Delta_g$ and the normalization condition (\ref{eq301}). Thus, the
inequality (\ref{eq003})
is proved. \\

Since $\{\theta_0, \cdots, \theta_d\}$ is an orthonormal basis of
$\cH_{g}$, we have \beq \| \ba_{t, \bb}\|^2=\int_M\;\langle \ba_{t,
\bb}, \Theta\rangle^2\,\oo_{g}^n\leq C\;\int_M\;\langle \ba_{t,
\bb}, \Theta\rangle^2\,\oo_{t, \bb}^n,\label{eq045} \eeq where we
used the fact that $\|\varphi\|_{W^{2, k+4}(M)}\leq C\,\ee_0$ when
$t$ small by Lemma \ref{lem2020}. The assumption $\Phi(t, \bb)=0$
together with (\ref{eq045}) and (\ref{eq003}) implies that \beqs
\int_M\;\langle \ba_{t, \bb}, \Theta\rangle^2\,\oo_{t,
\bb}^n&=&\int_M\;\Big(
\langle \Xi_{t, \bb}, \Te\rangle-\te_{t, \bb}\Big)\langle \Xi_{t, \bb}, \Te\rangle\,\oo_{t, \bb}^n\\
&\leq& C\;\ee_0\;\| \ba_{t, \bb}\|\cdot\|\langle   \ba_{t, \bb},
\Theta\rangle
\|_{L^2(\oo_g)}\\
&\leq &C\,\ee_0\, \int_M\;\langle   \ba_{t, \bb},
\Theta\rangle^2\,\oo_{t, \bb}^n.\eeqs Thus, if $\ee_0$ is small
enough  we have $ \Xi_{t,
\bb}=0$. The lemma is proved.\\

\end{proof}

Thus, the first part of Theorem \ref{theo001} and Corollary
\ref{cor001} follow directly  from Lemma \ref{lem003}.\\
\end{proof}

 Observe that we can expand the
function $\Phi(t, \bb)$ with respect to $t$ at $t=0:$
$$\Phi(t, \bb)=a_1(\bb)t+a_2(\bb)t^2+a_3(\bb)t^3+\cdot+a_m(\bb)t^m+O(t^{m+1}), $$
where $a_j(\bb)$ are the coefficients of $t^j$.  We want to ask what
kinds of K\"ahler metric
exists if we only assume the first several terms of   $a_i(\bb)$ vanish.    \\

\begin{cor}\label{cor015}Let $\oo_g$ be a constant scalar curvature metric. There
are two constants $\ee, C>0$ such that for any    harmonic form
$\bb\in \cH^{1, 1}(M) $ with unit norm and  \beq
a_1(\bb)=a_2(\bb)=\cdots=a_m(\bb)=0,\label{eq020}\eeq $M$ admits a
K\"ahler metric $\oo_{t, \bb}\in [\oo_t+t\bb]$ for $t\in (0, \ee_0)$
satisfying \beq \|\s(\oo_{t, \bb})-c_0(t)\|_{C^k(M)}\leq Ct^{\frac
{m+1}{2}}.\label{eq022}\eeq

\end{cor}

\begin{proof}We follow the notations in Lemma \ref{lem003}. By the
assumption (\ref{eq020}), there are two constants $\ee_0, C>0$ such
that for any $t\in (0, \ee_0)$ we have \beq |\Phi(t, \bb)|\leq
Ct^{m+1}.\label{eq021}\eeq By equality (\ref{eq004}) and
(\ref{eq003}) we have \beqs \Big|\Phi(t, \bb)-\int_M\;\langle
\Xi_{t, \bb}, \Te\rangle^2\;\oo_{t, \bb}^n\Big|&\leq& C\,t\|\Xi_{t,
\bb}\| \cdot\|\langle   \ba_{t, \bb}, \Theta\rangle \|_{L^2(\oo_g)}\\
&\leq &C\,t\, \int_M\;\langle   \ba_{t, \bb},
\Theta\rangle^2\,\oo_{t, \bb}^n \eeqs where we used  (\ref{eq045})
in the last inequality. Thus, there is a constant $\ee_0>0$ such
that for any $t\in (0, \ee_0)$ we have
$$\int_M\;\langle   \ba_{t, \bb},
\Theta\rangle^2\,\oo_{t, \bb}^n\leq C\cdot \Phi(t, \bb)\leq C\cdot
t^{m+1}$$ and hence  \beq \sum_{i=1}^d \,c_i(t)^2=\int_M\;\langle
\Xi_{t, \bb}, \Te\rangle^2\;\oo_{g}^n\leq C \int_M\;\langle  \Xi_{t,
\bb}, \Te\rangle^2\;\oo_{t, \bb}^n\leq Ct^{m+1}.\eeq This implies
that for each $i$ when $t$ is small, $|c_i(t)|\leq Ct^{\frac
{m+1}{2}}.$ Since $\oo_{t, \bb}$ is a solution of (\ref{eq0010}), we
have
$$\|\s(\oo_{t, \bb})-c_0(t)\|_{C^k(M)}=\|\sum_{i=1}^d\,c_i(t)\te_i\|_{C^k(M)}
\leq Ct^{\frac {m+1}{2}}.$$ The corollary is proved.

\end{proof}

Now we want to compute the coefficients of $t$ in  the expansion of
the function $\Phi$. Let $\oo_g$ be a constant scalar curvature
metric on $M$ and $(\varphi_{t, \bb}, \td \Xi_{t, \bb})$ the
solution of (\ref{eq0010}). Since the operator
$$\LL_g: \cH_{g, k+4}^{\perp}\ri \cH_{g, k }^{\perp} $$
is self-adjoint and invertible, we  denote by $\GG_g=\LL_g^{-1}$ the
inverse operator of $\LL_g.$ Without loss of generality, we can
assume that $\bb$ is traceless with respect to the metric $g$.
Otherwise, we can consider the metric $(1+t\cdot\tr_g\bb)\oo_g$
which still has constant scalar curvature. Let  $\cH_{0}^{1, 1}(M)$
be the space of traceless harmonic $(1, 1)$ form with respect to the
metric $g$ on $M.$  Computing the first derivative of
$S(t):=s(\oo_{t, \bb})-\langle \td \Xi_{t, \bb}, \td \Te\rangle$
with respect to $t$, we have\\

\begin{lem}\label{lem004}For $\bb\in \cH_0^{(1, 1)}$, we have the following:
\beqs  \langle \td \ba'(0), \td \Te\rangle&=&-\td\Pi_g(R_{i\bar j}\bb_{j\bar i}),\\
 \varphi'(0)&=&-\GG_g\td\Pi_g^{\perp}(R_{i\bar j}\bb_{j\bar i}),\\ c_0'(0)&=&
 \frac 1{V_g}\int_M\;R_{i\bar j}\bb_{j\bar i}\,\oo_g^n,\eeqs where we write
$f'(t)=\pd {f}t$ for simplicity.
\end{lem}
\begin{proof} Since $ S(t)=0$ for $t\in (0, \ee_0)$,  we have
 \beq 0=  S'(t)=-\Delta_{t}^2\varphi'(t)-R_{i\bar
j}(t)\varphi'_{j\bar i}(t) -R_{i\bar j}(t)\bb_{j\bar i}-\langle \td
\ba'_{t, \bb}(t), \td \Te\rangle. \label{eq007}\eeq Projecting to
the space $\cH_g$ when $t=0$ we have
$$0=\td\Pi_g(  S'(t))(0)=-\td\Pi_g(R_{i\bar j}\bb_{j\bar i})-\langle \td \ba_{t, \bb}'(0),
\td \Te\rangle,$$  which implies that
$$\langle \td \ba_{t, \bb}'(0),
\td \Te\rangle=-\td\Pi_g(R_{i\bar j}\bb_{j\bar i}).$$  On the other
hand, we project (\ref{eq007}) to the space $\cH_{g, k }^{\perp} $
and we have
$$0=\td\Pi_g^{\perp}( S'(t))(0)=-\LL_g \varphi'(0)-\td \Pi_g^{\perp}(R_{i\bar j}\bb_{j\bar i}).$$
This together with $\varphi'(0)\in\cH_{g, k+4 }^{\perp}$ implies
that
$$\varphi'(0)=-\GG_g\td\Pi_g^{\perp}(R_{i\bar j}\bb_{j\bar i}).$$
Now we calculate $c_0'(t).$ Note that $c_0(t)$ only depends on the
K\"ahler class $[\oo_g+t\bb]$, we compute it using the metric
$\oo_t=\oo_g+t\bb.$ Since $\pd
{}t\oo_t^n\Big|_{t=0}=\tr_{\oo_g}\bb\;\oo_g^n=0,$ we have $V_t'=0$
and
$$c_0'(0)=\frac 1{V_g}\int_M\;\pd {}ts(\oo_t) \Big|_{t=0}\,\oo_g^n=
\frac 1{V_g}\int_M\; R_{i\bar j}\bb_{j\bar i}\,\oo_g^n.$$ The lemma is proved.\\

\end{proof}

\begin{cor}\label{cor012} If $\bb\in \cH_0^{1, 1}(M)$, then the function $\Phi$
can be expanded as \beq \Phi(t, \bb)=t^2\int_M\;(\Pi_g(R_{i\bar
j}\bb_{j\bar i}))^2+O(t^3).\label{eq041}\eeq
\end{cor}
\begin{proof}
Since $\te_{t, \bb}(0)=0$ and $\Xi_{t, \bb}(0)=0,$ we have
$\Phi'(0)=0.$ Direct calculation shows   \beq
\Phi''(0)=\int_M\;2\te_{t, \bb}'(0)\langle \ba_{t, \bb}'(0),
  \Theta\rangle\,\oo_g^n.\label{eq048}\eeq
Taking the derivative with respect to $t$, we have \beq \i\bar \p
\te_{t, \bb}'(0)=\Big(i_{X_{t, \bb}' }\oo_{t, \bb}+i_{X_{t,
\bb}}\oo_{t,
\bb}'\Big)\Big|_{t=0}=\sum_{k=1}^d\,c_k'(0)\,i_{X_k}\oo_g=\i\sum_{k=1}^d\,c_k'(0)\bar
\p\theta_k,\nonumber\eeq which implies that \beq \te_{t,
\bb}'(0)=\sum_{k=1}^d\;c_k'(0)\theta_k=\langle \ba_{t, \bb}'(0),
  \Theta\rangle.\label{eq0021}\eeq
This together with the equality (\ref{eq048}) and Lemma \ref{lem004}
implies that
$$\Phi''(0)=2\int_M\;(\langle \ba_{t, \bb}'(0),
  \Theta\rangle)^2\,\oo_g^n=2\int_M\;(\Pi_g(R_{i\bar j}\bb_{j\bar i}))^2.$$
The corollary is proved.

\end{proof}

If $\oo_g$ is a K\"ahler-Einstein metric, the first term of the
right hand side of (\ref{eq041}) automatically vanishes. In this
case, it is not difficult to expand $\Phi(t, \bb)$ for more terms.

\begin{lem}\label{lem005}If $\bb\in \cH_0^{1, 1}(M)$  and
 satisfies $R_{i\bar j}\bb_{j\bar i}=0,$ then we have \beqs
 \langle \td \ba''_{t, \bb}(0), \td
\Te\rangle&=&\td\Pi_g(2R_{i\bar j}\bb_{j\bar k}\bb_{k\bar i}),\\
\varphi''(0)&=&\GG_g\td\Pi_g^{\perp}(2R_{i\bar j}\bb_{j\bar
k}\bb_{k\bar
i}),\\
c_0''(0)&=&\frac 1{V_g}\int_M\;2R_{i\bar j}\bb_{j\bar k}\bb_{k\bar
i}\,\oo_g^n.\eeqs
\end{lem}
\begin{proof}
Following the proof of Lemma \ref{lem004} we have \beqs  S''(t)
&=&(\bb_{i\bar j}+\varphi'_{i\bar j})(\Delta_t \varphi')_{j\bar
i}+\Delta_t((\bb_{i\bar j}+\varphi'_{i\bar j}) \varphi'_{j\bar
i})-\Delta_t^2 \varphi''\\&&+(\Delta_t\varphi')_{i\bar
j}\varphi'_{j\bar i}-R_{i\bar j}(t) \varphi''_{j\bar i}+2R_{i\bar
j}\varphi'_{j\bar k}(\bb+\Na^2\varphi')_{k\bar
i}+(\Delta\varphi')_{i\bar j}\bb_{j\bar i}\\&&+R_{i\bar j}\bb_{j\bar
k}(\bb+\Na^2D_t\varphi)_{k\bar i}+R_{i\bar j}\bb_{k\bar
i}(\bb+\Na^2D_t\varphi)_{j\bar k}-\langle \td \ba''_{t, \bb}, \td
\Te\rangle.\eeqs Thus, projecting to $\cH_g$ and $\cH_{g, k
}^{\perp}$ we have \beqs 0&=&\td\Pi_g(
 S''(t))(0)=\td\Pi_g(2R_{i\bar j}\bb_{j\bar k}\bb_{k\bar
i})-\langle \td \ba''_{t, \bb}(0), \td \Te\rangle,\\
0&=&\td\Pi_g^{\perp}(  S''(t))(0)=-\LL_g\varphi''+\td
\Pi_g^{\perp}(2R_{i\bar j}\bb_{j\bar k}\bb_{k\bar i}). \eeqs
Moreover, we calculate $c_0''(0)$ as in the proof of Lemma
\ref{lem004}
$$c_0''(0)=\frac
1{V_g}\int_M\;R(\oo_g+t\bb)''\Big|_{t=0}\;\oo_g^n=\frac 1{V_g}
\int_M\;2R_{i\bar j}\bb_{j\bar k}\bb_{k\bar i}\,\oo_g^n.$$ The lemma
is proved.

\end{proof}

\begin{cor}\label{cor014}
If $\oo_g$ is a K\"ahler-Einstein metric  and  $\bb\in \cH_0^{1,
1}(M)$, then we have
$$\Phi(t, \bb)=t^4\int_M\;(\Pi_g(R_{i\bar j}\bb_{j\bar k}\bb_{k\bar
i}))^2\,\oo_g^n+O(t^5).$$
\end{cor}
\begin{proof}By Lemma \ref{lem004}, we have
$$\langle \td \Xi_{t, \bb}'(0), \td \Te\rangle=\varphi'(0)=c_0'(0)=0.$$
Thus, the equality (\ref{eq0021}) implies that $\te'_{t, \bb}(0)=0$
and by direct calculation we have \beqs
\Phi_t'''(0)&=&3\int_M\;\Big(\te_{t, \bb}''(0) \langle \ba_{t,
\bb}'(0),
  \Theta\rangle+\te_{t, \bb}'(0)\langle  \ba_{t, \bb}''(0),
  \Theta\rangle\Big)\;\oo_g^n=0. \eeqs
On the other hand, by Lemma \ref{lem005} we have \beq \i\bar
\p\te''_{t, \bb}(0)=i_{X_{t,
\bb}''(0)}\oo_{g}=\sum_{k=1}^d\,c_k''(0)i_{X_k}\oo_g=\i\bar
\p\Big(\sum_{k=1}^d\,c_k''(0)\theta_k\Big), \eeq which implies that
\beq \te''_{t, \bb}(0)=\langle \Xi_{t, \bb}''(0),
\Te\rangle.\label{eq005}\eeq Thus, by tedious calculation we have
$$\Phi^{(4)}_t(0)=6\int_M\;\te_{t, \bb}''(0) \langle \ba_{t,
\bb}''(0),
  \Theta\rangle\;\oo_g^n=24\int_M\;(\Pi_g(R_{i\bar j}\bb_{j\bar k}\bb_{k\bar
i}))^2\,\oo_g^n.$$ The corollary is proved.

\end{proof}

\subsection{Varying complex
structures}\label{Sec002}

In this section, we will consider the deformation of   constant
scalar curvature metrics when the complex structure varies. Let $(M,
J, g, \oo_g)$ be a compact K\"ahler manifold $(M, J)$ with a
K\"ahler metric $g$ and the associate K\"ahler form $\oo_g.$ Let
$J_t$ be a smooth family of complex  structures with $J_0=J$. By
Kodaira's theorem in \cite{[Kod1]} there exists a smooth family of
K\"ahler metric $g_t$ with $g_0=g $ which is compatible with the
complex structure $J_t$ for small $t$. Let $\oo_t$ be the associate
K\"ahler form of $g_t$ with respect to the complex structure $J_t$.
The triple $(J_t, g_t, \oo_t )$ is called a complex deformation of
$(J, g, \oo_g).$ Given a complex deformation $(J_t, g_t, \oo_t )$,
we want to know whether there exists a constant scalar curvature
metric in the K\"ahler class $([\oo_t], J_t)$ if  we assume that
$\oo_g$ is a constant scalar curvature metric on $(M,
J)$. \\

 Since $g$ is a constant
scalar curvature metric, the identity component $G$ of the isometry
group of $(M, g)$ is a maximal compact subgroup of $\Aut(M, g)$ by
Lichnerowicz-Matsushima theorem. In general the action of the group
$G$ may not extend to  $(M, J_t)$. We follow the idea of
Rollin-Simanca-Tipler in \cite{[RST]} to assume that a compact
connected subgroup $G'$ of $G$ can extend to $(M, J_t)$ and $G'$
acts holomorphically on the complex deformation $(J_t, g_t, \oo_t)$.
We denote by $\cB_{G'}$  the space of  complex deformations $(J_t,
g_t, \oo_t)$ which allow the holomorphic action of $G'$. We denote
by $W^{2, k}_{G'}(M)$ the subspace of $G'$-invariant functions in
$W^{2, k}(M)$ and
 $\cU$  a neighborhood
of the origin in $W^{2, k}_{G'}(M)$.
 For any $\varphi\in \cU$, we
compute the expansion of the scalar curvature of the metric $\oo_{t,
\varphi}=\oo_t+\pt\varphi$ at $(t, \varphi)=(0, 0):$

\begin{lem}\label{lem009} Suppose that $\p \oo_t/\p t=\eta_t.$ We
have
$$\s(\oo_{t, \varphi})=\s(\oo_g)-\LL_g\varphi-t\Big(\Delta_g\tr_{\oo_{g}}(\eta
+S(\varphi))+R_{i\bar j}(\eta+S(\varphi))_{j\bar i} +\tr_g (S
\log\det g)\Big)+Q,$$ where $Q$ collects all the higher order terms
and the operator $S$ is given by $S=\frac 12dJ_t'(0)df$.
\end{lem}
\begin{proof}
For any smooth function $f$, we define the operator
$$S_t(f):=\pd {}t\pt (f)=\frac 12dJ'_tdf,$$
where we used the equality $\pt=\frac 12dJ_td.$  Note that
$$\pd {}t\oo_{t,  \varphi}=\eta_t+ S_t(\varphi),\quad
 D_{\varphi}\oo_{t,  \varphi}(\psi)=\pt\psi.$$
The derivatives of the scalar curvature are given by \beqs \pd
{}{t}\s(\oo_{t, I, \varphi})&=&- (\eta_{i\bar j}+S_{t, i\bar
j}(\varphi)) R_{j\bar i}-g^{i\bar j}S_{t, i\bar j}(\log\det g
)-\Delta_t\tr_{\oo_{t, \varphi}}(\eta+ S_t(\varphi))\\
D_{\varphi}\s(\oo_{t, I, \varphi})(\psi)&=&-R_{i\bar j}\psi_{j\bar
i}-\Delta_t^2\psi. \eeqs Thus, the lemma follows directly.

\end{proof}

As in Section \ref{Sec001}, we define $\frak g$(resp. $\frak g'$)
the Lie algebra of $G$(resp. $G'$), and $\frak g_0$ (resp. $\frak
g_0'$) the ideal of Killing vector fields with zeros in $\frak g$
(resp. $\frak g'$). The center of $\frak g_0$(resp. $\frak g_0'$) is
denoted by $\frak z_0$(resp. $\frak z_0'$). Each element of $\frak
z_0$(resp. $\frak z_0'$) is of the form $J\Na f$
  for a  {$G$(resp. $G'$)-invariant},  real-valued function $f$.
Let $\cH_g^{\frak g_0'}$ (resp. $\cH_g^{\frak z_0'}$) the space of
holomorphic potentials of the Killing vector fields in $\frak g_0'$
(resp. $\frak z_0'$) and it is easy to see that the space
$\cH_g^{\frak z_0'}$ is identified to the $G'$-invariant holomorphic
potentials of $\cH_g^{\frak g_0'}$.  Using the $L^2$ inner product
induced by $g$, the space $W_{G'}^{2, k}(M)$ has the orthogonal
decomposition
$$W_{G'}^{2, k}(M)=\cH_g \oplus  \cH_{g, k}^{\perp}$$
where $\cH_g=\RR\oplus\cH_g^{\frak z_0'}$ and we assume $\cH_g$ is
spanned by an orthonormal basis $\{\te_0, \te_1, \cdots, \te_d\}$
where $\te_0=1$ with respect to the induced $L^2$ norm of the metric
$g$.
 Let $\td \Pi_g$
and $\td \Pi_g^{\perp}$ be the $L^2$-orthogonal projection onto
$\cH_g$ and $\cH_{g, k}^{\perp}$ respectively.  With these
notations,
we have the result:\\

\begin{theo}\label{theo015} Let $g$ be a constant scalar curvature
metric on $M$ with \beq   \ker\LL_g\cap W_{G'}^{2, k}\subset
\RR\oplus\cH_g^{\frak z_0'}.\label{eq050} \eeq For any $(J_t, g_t,
\oo_t) \in \cB_{G'}$, there is a constant $\ee_0>0$ and a smooth
function $\Psi: \cB_{G'}\ri \RR$ such that if $\Psi(J_t, g_t,
\oo_t)=0$ for some $t\in (0, \ee_0)$,
  then $M$ admits a $G'$-invariant constant scalar curvature metric
  in  $[\oo_t]$ with respect to $J_t$.\\
\end{theo}

\begin{proof}First, we want to find the solution $(\varphi, \td \Xi)\in
\cH_{g, k+4}^{\perp}\times \RR^{d+1}$ of the equation \beq
\s(\oo_{t, \varphi})=\langle\td \ba, \td \Te\rangle,
\label{eq015}\eeq where $\td \Te=(\te_0, \te_1, \cdots, \te_d)$.  As
in the proof of Lemma \ref{lem2020}, we can use the implicit
function theorem and Lemma
\ref{lem009} to show that\\

\begin{lem}\label{lem010} Suppose that the condition (\ref{eq050}) holds.
For any $(J_t, g_t, \oo_t) \in \cB_{G'}$, there exist $C, \ee_0>0$
such that for all $t\in (0, \ee_0)$ there is a solution
$(\varphi_{t}, \td \Xi_{t })\in \cH_{g, k+4}^{\perp}\times
\RR^{d+1}$ which satisfies the equation (\ref{eq015}) and \beq
\|\varphi_{t}\|_{W^{2, k+4}(M)}\leq C\ee_0,\quad \|\td \ba_{t}\|\leq
C\ee_0.\label{eq042}\eeq
\end{lem}
\begin{proof}The linearization of the operator $\td \Pi_g^{\perp}\s(\oo_{t, \varphi}):
(-\ee, \ee)\times \cH_{g, k+4}^{\perp}\ri \RR$ at $(t, \varphi)=(0,
0)$ is given by
$$D_{\varphi}\td \Pi_g^{\perp}\s(\oo_{t, \varphi})|_{(0, 0)}(\psi)=-\LL_g\psi:
\cH_{g, k+4}^{\perp}\ri  W_{G'}^{2, k},$$ which is invertible from
$\cH_{k+4}^{\perp}$ to $\cH_{g, k}^{\perp}$ if and only if the
condition (\ref{eq050}) holds. Thus, the lemma follows directly from
the implicit function theorem.

\end{proof}

Let $\xi_i(1\leq i\leq d)$ be the Killing vector fields in $\frak
z_0$ with the holomorphic potentials $\theta_i(1\leq i\leq d).$
 Since $(J_t, g_t, \oo_t)\in \cB_{G'}$,
 the vector fields  $X_i^t:=J_t\xi_i+\i \xi_i$ are holomorphic on $(M, J_t)$ and the holomorphic potential of $X_i^t $ with
respect to $\oo_{t, \varphi_t}$ is given by a real-valued function
$\theta_{i}^t$ satisfying \beq i_{X_i^t}\oo_{t, \varphi_t}=\i\bar
\p_t\theta_i^t,\quad \int_M\;\te_i^t\,\oo_{t,
\varphi_t}^n=0.\label{eq302}\eeq For the vector $\td \Xi_t=(c_0(t),
c_1(t), \cdots, c_d(t))\in \RR^{d+1}$ obtained in Lemma
\ref{lem010}, we define the holomorphic vector field
$$X_t=\sum_{i=1}^d\,c_i(t)X_i^t \in \frak h_0(M, J_t).$$
 Let $\theta_t$ be the holomorphic
potential of $X_t$ with respect to $\oo_{t, \varphi_t}$ and
$$\Te=(\te_1, \cdots, \te_d),\quad \Xi_t=(c_1(t), \cdots, c_d(t)),$$
where $c_i(t)$ are the entries of $\td \Xi_t.$

\begin{lem}\label{lem012}If $(J_t, g_t, \oo_t)\in \cB_{G'}$
satisfies \beq \|J_t-J_0\|_{C^1(M)}\leq C\ee_0,\quad t\in (0,
\ee_0),\label{eq006}\eeq  then there is a constant $C_1>0$ such that
for all $t\in (0, \ee_0)$ we have \beq \|\te_{t}-\langle \ba_{t},
\Theta\rangle\|_{L^2(\oo_g)}\leq C_1\,\ee_0 \|\ba_{t}\|.
\label{eq033} \eeq

\end{lem}
\begin{proof}
Define the vector field $\hat X_{t}=\sum_{k=1}^d \,c_k(t)X_i\in
\frak h_0(M, J)$ where $c_k(t)$ is given by Lemma \ref{lem010}.  By
definition, we have
$$i_{\hat X_{t}}\oo_g=\i\bar \p \langle   \ba_{t},   \Theta\rangle,
\quad i_{X_{t}}\oo_{t, \varphi_t}=\i\bar \p_t \theta_{t},$$ where
$\bar \p$ denotes the  operator on $(M, J)$. We want to compute the
difference of the two functions $\theta_{t}$ and $\langle \ba_{t},
\Theta\rangle:$
 \beqn \i\bar \p(\langle \Xi_{t }, \Te
\rangle-\te_{t})&=&i_{\hat X_{t }}\oo_g- i_{X_t}\oo_{t,
\varphi_t}+\i(\bar \p_t-\bar
\p)\theta_{t}\nonumber\\
&=&\sum_{k=1}^d\;c_k(t)(i_{X_k}\oo_g-i_{X_k^t}\oo_{t, \varphi_t})+\i
(\bar \p_t-\bar \p)\theta_{t}. \label{eq036}\eeqn Note that the
estimate $\|\oo_g-\oo_{t, \varphi_t}\|_{W^{2, k+2}(M)}\leq C\ee_0$
obtained in Lemma \ref{lem010} implies  \beqn \|  \p
\Big(i_{X_k}\oo_g-i_{X_k^t}\oo_{t}\Big)\|_{C^0}&=&\|i_{  \p
(X_k-X_k^t)}\oo_g+i_{  \p X_k^t}(\oo_g-\oo_{t, \varphi_t})+i_{X_k^t}
\p(\oo_g-\oo_{t,
\varphi_t})\|_{C^0}\nonumber\\
&\leq &C\ee_0, \label{eq043} \eeqn where we used the estimates
$$ \|  \p (X_k^t-X_k)\|_{C^0}=\|  \p
(J_t-J_0)\xi_k\|_{C^0}\leq C\ee_0,\quad t\in (0, \ee_0).$$ Now we
estimate $\theta_{t}$. Note that we have
$$\Delta_{\oo_{t, \varphi_t}}\theta_{t}=\i\p_t(i_{X_t}\oo_{t, \varphi_t})=
\i\sum_{k=1}^d\;c_k(t)\p_t(i_{X_t}\oo_{t, \varphi_t})$$ and
$\|\oo_{t, \varphi_t}-\oo_g\|_{C^{2, \al}}\leq C\ee_0$ if we choose
$k$ sufficiently large in Lemma \ref{lem010}, there is a constant
$C>0$ independent of $t$ such that \beq \|\theta_{t}\|_{C^2(M,
\oo_g)}\leq C\|\Xi_{t}\|.\label{eq046}\eeq Therefore, we have \beqn
\Big|\p(\bar \p_t-\p)\te_{t}\Big|=\frac 12\Big|\p
(J_t-J_0)d\te_{t}\Big|\leq C \ee_0\cdot\|\Xi_{t}\|,\quad t\in (0,
\ee_0), \label{eq040}\eeqn where we used the equality $\bar
\p_tf=\frac 12(df-\i J_tdf)$  and the inequality (\ref{eq046}).
Combining the estimates (\ref{eq036})(\ref{eq043}) and
(\ref{eq040}), we have
 \beq
\Big|\Delta_g(\langle \Xi_{t}, \Te \rangle-\te_{t})\Big|\leq
C\ee_0\,\cdot\|\Xi_{t}\|.\nonumber\eeq This together the eigenvalue decomposition
and the normalization condition (\ref{eq302}) gives (\ref{eq033}).
 The lemma is proved.

\end{proof}

Now we define the function $\Psi: \cB_{G}\ri \RR$ by
$$\Psi(J_t, g_t, \oo_t)=\int_M\; X_t h_{t, \varphi_t}\,\oo_{t, \varphi_t}^n=\int_M\;
\theta_t(\s(\oo_{t, \varphi})-c_0(t))\,\oo_{t, \varphi_t}^n,$$ where
$c_0(t)$ is the average of $s(\oo_{t, \varphi_t})$ and $h_{t,
\varphi_t}$ is given by $s(\oo_{t,
\varphi_t})-c_0(t)=\Delta_{\oo_{t, \varphi_t}} h_{ {t, \varphi_t}}.$
    As in Section \ref{Sec001}, we have the following result whose
    proof is omitted.

\begin{lem}\label{lem007}There exists $\ee_0>0$ such that
if the complex deformation $(J_t, g_t, \oo_t)\in \cB_{G}$ satisfies
$\Psi(J_t, g_t, \oo_t)=0$ for some $t\in (0, \ee_0)$,  then $\oo_{t,
\varphi_{t}}$ is a constant scalar curvature metric with respect to
the complex structure $J_t.$

\end{lem}
 Theorem \ref{theo015} then follows from the above results.

\end{proof}

\section{Deformation of K\"ahler-Ricci solitons}
 Let $(M, J)$ be a compact
K\"ahler manifold with a K\"ahler Ricci soliton $g_{KS}$ with
respect to the holomorphic vector field $X:$
$$Ric(\oo_{KS})-\oo_{KS}=\pbp \theta_X$$
where $\te_X$ is the holomorphic potential of $X$ with respect to
$\oo_{KS}.$ We would like to ask whether we can perturb the K\"ahler
Ricci soliton under complex deformation of the complex structure.
Inspired by the discussion before, for any K\"ahler class $[\oo_g]
$we consider the metric  $\oo_{\varphi}\in [\oo_g]$ satisfying the
equation of extremal solitons \beq \s(\oo_{\varphi})-\un
\s=\Delta_{\varphi}\theta_X(\oo_\varphi). \label{eq101}\eeq By the
$\p\bar\p$-Lemma, we can easily check that

\begin{lem}\label{lem101}
If $\oo_g\in 2\pi c_1(M)$ satisfies the equation $(\ref{eq101})$
with respect to a holomorphic vector field $X$, then $\oo_g$ is a
K\"ahler-Ricci soliton with respect to $X.$
\end{lem}

  By the equation (\ref{eq101}), if
$[\oo_g]$ admits an extremal soliton $\oo_{\varphi}$ and the Futaki
invariant vanishes on $[\oo_g]$, then $\oo_{\varphi}$ must be a
constant scalar curvature metric. In fact,
$$f(X, [\oo_0])=\int_M\;\theta_X(\varphi)\Delta_{\varphi}\theta_X(\varphi)\,\oo_{\varphi}^n=0$$
implies that $\theta_X(\varphi)$ is a constant. \\

\begin{theo}\label{theo101}If $\oo_g$ be a K\"ahler Ricci soliton with respect to $X$
on $M,$ then for any $\bb\in \cH^{1, 1}(M)$
there is an extremal soliton in the K\"ahler class $[\oo_0+t\bb]$
for small $t$.

\end{theo}

\begin{proof} We follow Lebrun-Simanca's arguments in
\cite{[LB1]}\cite{[LB2]}.  Let $g$ be a K\"ahler-Ricci soliton. By
Theorem A in the appendix of \cite{[TZ1]} the identity component $G$
of the isometry group of $(M, g)$ is a maximal compact subgroup of
the automorphism group $\Aut(M).$ As in previous sections, we let
$W^{2, k}_G$ be the real $k$-th Sobolev space of $G$-invariant
real-valued functions in $W^{2, k}.$ Let $\frak g$ the Lie algebra
of $G$ and $\frak z\subset\frak g$ denote the center of $\frak g.$
We denote by $\frak g_0$ the ideal of Killing vector fields with
zeros and $\frak z_0=\frak z\cap \frak g_0$.  By Lemma A.2 in the
appendix of \cite{[TZ1]}, each element of $\frak z_0$ is of the form
$J\Na f$, where $f$ is a $G$-invariant real-valued function
satisfying the equation
$$\cL_{g}(f)=f_{\bar i\bar j}dz^{\bar i}\otimes dz^{\bar j}=0.$$
We choose a basis $\{\xi_1, \cdots, \xi_d\}$
 of $\frak z_0$ such that the functionals $\{\te_0, \te_1, \cdots,
 \te_d\}$ where $\te_0=1$ and $\te_i(1\leq i\leq d)$ is the holomorphic potential
 of the holomorphic vector fields   $X_i=J\xi_i+\i\xi_i $  are orthonormal
with respect to the $L^2$ inner product
$$\langle f, g\rangle_{L^2(\oo_g)}=\frac 1{V_g}\int_M\;fge^{\theta_X}\,\oo_g^n,\quad f, g\in C^{\infty}(M, \RR),$$
where $V_g$ is the volume of $(M, g).$ Using this product, the space
$W^{2, k}_G$ has a decomposition $W^{2, k}_G=\cH_g\oplus \cH_{g,
k}^{\perp},$ where $\cH_g$ is spanned by the set $\{\theta_0, \te_1,
\cdots, \te_d\}$ over $\RR.$  We define the associate project
operator $\Pi_g$ and $\Pi_g^{\perp},$ and we can assume that $X_1=X$
which defines the K\"ahler-Ricci soliton $\oo_g.$\\

Now we consider the equation for $\varphi\in \cU:$
$$S(t, \varphi):=\Pi_g^{\perp}\Pi_{\varphi}^{\perp}G_{\varphi}(\s(\oo_{t, \varphi})-\un \s(t))=0,$$
where $G_{\varphi}$ is the Green operator with respect to the metric
$\oo_{t, \varphi}$. If $\cU$ is small enough, $  S(t, \varphi)=0$ if
and only if $\oo_{t, \varphi}$ is an extremal soliton. We calculate
the variation of $S(t, \varphi)$ at $(t, \varphi)=(0, 0):$ \beqn
D_{\varphi} S(t, \varphi)|_{(0,
0)}(\psi)&=&-\Pi_g^{\perp}(D_{\varphi}\Pi_{\varphi})|_{(0,
0)}G_g(\s(\oo_g)-\un
\s)+\Pi_g^{\perp}D_{\varphi}(G_{\varphi}(\s(\oo_{t, \varphi})-\un
\s))|_{(0, 0)}.\nonumber\\\label{eq102} \eeqn  Since $g$ is a
K\"ahler Ricci soliton, we have $G_g(\s(\oo_g)-\un \s)=\theta_X$.
Note that
$$\Pi_{\varphi} \te_X=\sum_{i=0}^d\,\langle \theta_{i, \varphi},
\theta_X\rangle_{L^2(\oo_{t, \varphi})} \te_{i, \varphi},$$ where
$\te_{i, \varphi}$ is an orthonormal basis of $\cH_g.$ Now we choose
the functions
$$\te_{0, \varphi}=1, \quad \theta_{i, \varphi}=\frac {\td \te_{i,
\varphi}}{\|\td \te_{i, \varphi}\|_{L^2(\oo_{\varphi})}},\quad 1\leq
i\leq d,$$ where $\td \te_{i, \varphi}$ are defined by the
equalities $ i_{X_i}\oo_{t, \varphi}=\i\bar\p \td \te_{i, \varphi}$
such that $\{\theta_{0, \varphi}, \cdots, \te_{d, \varphi}\}$ forms
an orthonormal basis of $\cH_{\varphi}, $ which is the space defined
similar to $\cH_g$ using the metric $\oo_{\varphi}.$
 Thus,
we have
 \beqs
-\Pi_g^{\perp}(D_{\varphi}\Pi_{\varphi})|_{(0, 0)}G_g(\s(\oo_g)-\un
\s)&=&-\Pi_g^{\perp}(D_{\varphi}\Pi_{\varphi})|_{(0, 0)}\theta_X\\&
=&-\langle \frac {\te_X}{\|\te_X\|_{L^2}}, \te_X
\rangle_{L^2(\oo_g)}\Pi_g^{\perp}\frac
1{\|\te_X\|_{L^2}}D_{\varphi}\td \theta_{1, \varphi}|_{(0, 0)}\\
&=&-\Pi_g^{\perp}D_{\varphi}\td \theta_{1, \varphi}|_{(0, 0)}. \eeqs
By the definition of $\td \te_{1, \varphi}$, we have
$$i_{X}D_{\varphi}\oo_{t, \varphi}|_{(0, 0)}=\i\bar \p D_{\varphi}\td \te_{1,
\varphi},$$ which implies that $X(\psi)=D_{\varphi}\td \te_{1,
\varphi}|_{(0, 0)}(\psi).$ Combining the above equalities, we have
\beq -\Pi_g^{\perp}(D_{\varphi}\Pi_{\varphi})|_{(0,
0)}G_g(\s(\oo_g)-\un \s)=-\Pi_g^{\perp}X(\psi). \label{eq103} \eeq

Now we calculate the second term of the right hand side of
(\ref{eq102}). Let $A_{\varphi}=G_{\varphi}(\s(\oo_{t, \varphi})-\un
\s)$, we have
$$\Delta_{\varphi}A_{\varphi}=\s(\oo_{t, \varphi})-\un
\s.$$ Differentiating this equation with respect to $\varphi$ at
$(t, \varphi)=(0, 0)$, we have
$$-\psi_{i\bar j}\te_{X, j\bar i}+\Delta_gD_{\varphi}A_{\varphi}|_{(0, 0)}=
-\Delta_g^2\psi-R_{i\bar j}\psi_{j\bar i}.$$ Combining this with
(\ref{eq102}) we have \beqn
\Pi_g^{\perp}D_{\varphi}(G_{\varphi}(\s(\oo_{t, \varphi})-\un
\s))|_{(0, 0)}&=&-\Pi_g^{\perp}G_g\Big(\Delta_g^2\psi+R_{i\bar
j}\psi_{j\bar i}-\psi_{i\bar j}\te_{X, j\bar i}\Big).\label{eq104}
\eeqn Combining the equalities (\ref{eq102})-(\ref{eq104}), we have
\beqs D_{\varphi} S(t, \varphi)|_{(0,
0)}(\psi)&=&-\Pi_g^{\perp}\Big(G_g (\Delta_g^2\psi+R_{i\bar
j}\psi_{j\bar i}-\psi_{i\bar j}\te_{X, j\bar i})+X(\psi)\Big)\\
&=&-\Pi_g^{\perp} \Big(\Delta_g\psi+\psi+X(\psi)\Big) \eeqs where we
used the assumption that $g$ is a K\"ahler-Ricci soliton. Note that
by Lemma 2.2 in \cite{[TZ1]}  the function $\psi$ satisfies
$\Delta_g\psi+\psi+X(\psi)=0$ if and only if $\Pi_g^{\perp}\psi=0$
Thus, the operator
$$D_{\varphi} S(t, \varphi)|_{(0,
0)}: \cH_{g, k+2}^{\perp}\ri \cH_{g, k}^{\perp}$$ is invertible and
by the implicit function theorem there is a solution $\varphi_t\in
\cH_{g, k+2}^{\perp}$ satisfies the equation $ S(t, \varphi_t)=0$ when $t$ is small. The
theorem is proved. \\

\end{proof}

\begin{rem}
It is interesting to ask whether Theorem \ref{theo101} holds for any
extremal soliton $g$. To prove this, it suffices to show that any
function $\psi$ with
$$\psi_{ij\bar j\bar i}+\theta_{X, \bar i}\psi_{ik\bar k}=0$$ must
satisfy the equation $\psi_{\bar i\bar j}=0.$
\end{rem}
In fact, if $g$ is an extremal soliton, we have
 \beqs
  G_g (\Delta_g^2\psi+R_{i\bar
j}\psi_{j\bar i}-\psi_{i\bar j}\te_{X, j\bar
i})+X(\psi)&=&G_g(\Delta_g^2\psi+R_{i\bar j}\psi_{j\bar
i}-\theta_{i\bar j}\psi_{j\bar i}+\Delta_g(X\psi))\\
&=&G_g(\Delta_g^2\psi+R_{i\bar j}\psi_{j\bar i}+\s_{, \bar
i}\psi_{i}+\te_{X, \bar i}\psi_{ik\bar k})\\
&=&G_g(\psi_{ij\bar j\bar i}+\te_{X, \bar i}\psi_{ik\bar k}),\eeqs
where we used  the equality \beqs \Delta_g(X\psi)&=&\frac
12\Big((\te_{\bar i}\psi_i)_{j\bar j}+(\te_{\bar i}\psi_i)_{\bar
jj}\Big) =\theta_{\bar ij}\psi_{i\bar j}+\theta_{\bar
i}(\Delta\psi)_i\\
&=&\theta_{\bar ij}\psi_{i\bar j}+\te_{\bar i}(\psi_{ik\bar
k}-R_{i\bar j}\psi_{j})\\&=&\theta_{\bar ij}\psi_{i\bar j}+\te_{\bar
i}\psi_{ik\bar k}+\s_{, \bar i}\psi_{i}.
 \eeqs Here we used the extremal soliton equation in the last
 equality. \\

 Next, we use the similar method in Section \ref{Sec001} to consider
 the case when the complex structure varies   . Let $(g, \oo_g)$ is a
K\"ahler-Ricci soliton on $(M, J)$ and $(J_t, g_t, \oo_t)$  a
complex deformation of $(J, g, \oo_g)$. We assume $(J_t, g_t,
\oo_t)\in \cB_G$ where $G$ is the identity component of the isometry
group of $(M, g)$ and $\cB_G$ denotes all the $G$ invariant complex
deformation of $(J, g, \oo_g).$ With these notations, we have the
 result:

 \begin{theo}\label{theo102}  Let $(M, J, g,  \oo_{g})$ be a
 compact K\"ahler manifold with
 a K\"ahler-Ricci soliton
$(g, \oo_g)$.  For any  $(J_t, g_t, \oo_t) \in \cB_{G}$,  $M$ admits
a $G$-invariant extremal soliton in  $[\oo_t]$ with respect to $J_t$
for small $t.$
\end{theo}
\begin{proof} The proof is more or less the same as in Theorem
\ref{theo101}, and we only sketch it here. For any $(J_t, g_t,
\oo_t)\in \cB_{G}$, we consider the equation \beq S(t,
\varphi):=\Pi_g^{\perp}\Pi_{\varphi}^{\perp}G_{\varphi}(\s(\oo_{t,
\varphi})-\un \s(t))=0,\label{eq105}\eeq where $G_{\varphi}$ and
$\Pi_{\varphi}^{\perp}$ are the operators with respect to the metric
$\oo_{t, \varphi}=\oo_t+\pt \varphi$.
   Let $\{\xi_1, \cdots, \xi_d\}$ be a basis of
$\frak z_0$. Since $(J_t, g_t, \oo_t) \in \cB_{G}$, the vector
fields $\{X_1^t, \cdots, X_d^t\}$ where $X_i^t=J_t\xi_i+\i \xi_i$
are holomorphic vector fields on $(M, J_t)$ and form a basis of
$\frak h_0(M, J_t).$ Let $\td\theta_i^t(1\leq i\leq d)$ be the
holomorphic potentials of $X_i^t$ with respect to $\oo_{t, \varphi}$
and we assume that the set $\{\td\theta_0^t, \td \te_1^t, \cdots,
\td\te_d^t\}$ where $\td\te_0^t=1$ are orthonormal and spans the
space $\cH_{\varphi}.$ Differentiating the equation (\ref{eq105})
with respect to $\varphi$, we have \beqn D_{\varphi}S(t,
\varphi)|_{(0,
0)}(\psi)&=&-\Pi_g^{\perp}(D_{\varphi}\Pi_{\varphi})|_{(0,
0)}G_g(\s(\oo_g)-\un
\s)+\Pi_g^{\perp}D_{\varphi}(G_{\varphi}(\s(\oo_{t, \varphi})-\un
\s))|_{(0, 0)}.\nonumber\eeqn  Since $D_{\varphi}\oo_{t,
\varphi}|_{(0, 0)}(\psi)=\pbp\psi$ and $D_{\varphi}X^t|_{(0, 0)}=0,$
we still get the equality (\ref{eq103}). By the same calculation as
in Theorem \ref{theo101}, we have the operator \beqs D_{\varphi}S(t,
\varphi)|_{(0, 0)}(\psi) =-\Pi_g^{\perp}
\Big(\Delta_g\psi+\psi+X(\psi)\Big) \eeqs which is invertible from
$\cH_{g, k+2}^{\perp}$ to $\cH_{g, k}^{\perp}$. The theorem is
proved.

\end{proof}

Here we give an easy example on the existence of extremal solitons.

\begin{ex}Let $\pi: \hat M\ri M$ be the blowup of $M=\CC\PP^2$ at a point $p$. Then $\hat M$
has no K\"ahler-Einstein metrics but admits a K\"ahler-Ricci soliton
in $2\pi c_1(\hat M).$ Thus, $\hat M$ admits extremal solitons in
the K\"ahler class $2\pi c_1(\hat M)-t[E]$ for $t\in (0, \ee)$ where
$E=\pi^{-1}(p)$ is the exceptional divisor and $\ee>0$ is small.

\end{ex}

\noindent Haozhao Li,\\
 Department of Mathematics,  University of Science and Technology of China, Hefei, 230026, Anhui
province, China.
Email: hzli@ustc.edu.cn\\

\end{document}